\documentclass[a4paper]{article}
\usepackage[T1]{fontenc}
\usepackage[utf8]{inputenc}
\usepackage{color}
\usepackage{mathrsfs}
\usepackage{amsmath}
\usepackage{amsthm}
\usepackage{amssymb}
\usepackage{stackrel}
\usepackage{setspace}
\usepackage{esint}
\usepackage[authoryear]{natbib}
\usepackage[unicode=true,pdfusetitle,
 bookmarks=true,bookmarksnumbered=true,bookmarksopen=true,bookmarksopenlevel=1,
 breaklinks=true,pdfborder={0 0 0},backref=false,colorlinks=true]
 {hyperref}
\hypersetup{
 citecolor=blue,urlcolor=blue}
\usepackage{breakurl}

\makeatletter

  \theoremstyle{definition}
  \newtheorem*{example*}{\protect\examplename}
\theoremstyle{plain}
\newtheorem{thm}{\protect\theoremname}
  \theoremstyle{definition}
  \newtheorem{defn}[thm]{\protect\definitionname}
  \theoremstyle{plain}
  \newtheorem{lem}[thm]{\protect\lemmaname}
  \theoremstyle{plain}
  \newtheorem{prop}[thm]{\protect\propositionname}
  \theoremstyle{remark}
  \newtheorem*{rem*}{\protect\remarkname}
  \theoremstyle{plain}
  \newtheorem*{assumption*}{\protect\assumptionname}
  \theoremstyle{plain}
  \newtheorem{cor}[thm]{\protect\corollaryname}
  \theoremstyle{plain}
  \newtheorem*{prop*}{\protect\propositionname}

\@ifundefined{date}{}{\date{}}

\makeatother

  \providecommand{\assumptionname}{Assumption}
  \providecommand{\corollaryname}{Corollary}
  \providecommand{\definitionname}{Definition}
  \providecommand{\examplename}{Example}
  \providecommand{\lemmaname}{Lemma}
  \providecommand{\propositionname}{Proposition}
  \providecommand{\remarkname}{Remark}
\providecommand{\theoremname}{Theorem}

\begin{document}

\title{Local Asymptotic Normality of Infinite-Dimensional Concave Extended
Linear Models}

\author{K\={o}saku Takanashi\\
Faculty of Economics, Keio University}
\maketitle
\begin{abstract}
We study local asymptotic normality of M-estimates of convex minimization
in an infinite dimensional parameter space. The objective function
of M-estimates is not necessary differentiable and is possibly subject
to convex constraints. In the above circumstance, narrow convergence
with respect to uniform convergence fails to hold, because of the
strength of it's topology. A new approach we propose to the lack-of-uniform-convergence
is based on Mosco-convergence that is weaker topology than uniform
convergence. By applying narrow convergence with respect to Mosco
topology, we develop an infinite-dimensional version of the convexity
argument and provide a proof of a local asymptotic normality. Our
new technique also provides a proof of an asymptotic distribution
of the likelihood ratio test statistic defined on real separable Hilbert
spaces. 
\end{abstract}

\section{Introduction}

We develop an infinite-dimensional version of local asymptotic normality
and convexity arguments with non-differentiable objective functions
in M-estimation of concave extended linear models. A new approach
we propose is based on Mosco convergence that is weaker than uniform
convergence in the topological sense. Because of the strength of uniform
convergence, it does not fold in infinite dimensional circumstances.
In this paper, we give proofs of local asymptotic normality on a real
separable Hilbert space.

The basic set-up of the estimation problem we investigate is as follows.
Let $\mathscr{H}$ be a real separable Hilbert space with the identical
dual $\mathscr{H}^{*}=\mathscr{H}$. We denote the inner product and
the associated norm in $\mathscr{H}$ by $\left\langle \cdot,\cdot\right\rangle $
and $\left\Vert \cdot\right\Vert $ respectively. Let $\theta$ be
a parameter vector in a parameter set $\Theta$ such that $\Theta\subseteq\mathscr{H}$.
Suppose we have $n$ observations $Z_{1},\dots,Z_{n}$ that are realizations
of a random vector $Z$ on a arbitrary set $E$, and consider an M-estimator
of the unknown parameter vector $\theta$ such that

\begin{alignat}{1}
\hat{\theta}_{n} & =\arg\min_{\theta\in\Theta}\frac{1}{n}\sum_{i=1}^{n}\rho\left(\theta,Z_{i}\right),\label{eq:m-estimator}
\end{alignat}
where $\rho:\Theta\times E\rightarrow(-\infty,\infty]$ is a criterion
function. Define the empirical objective function in (\ref{eq:m-estimator})
as
\begin{equation}
F_{n}\left(\theta\right)\triangleq\frac{1}{n}\sum_{i=1}^{n}\rho\left(\theta,Z_{i}\right),\label{eq:empirical_object}
\end{equation}
and its population counterpart as
\begin{equation}
F_{0}\left(\theta\right)\triangleq\mathbb{E}_{Z}\left[\rho\left(\theta,Z\right)\right].\label{eq:population_object}
\end{equation}
We further suppose the minimization problem in (\ref{eq:m-estimator})
is corresponding to a ``concave extended linear model'', that is,
\begin{enumerate}
\item $\rho$ is a lower semi-continuous (l.s.c.) convex function with respect
to $\theta$ (it is not necessarily smooth, though),
\item $F_{0}\left(\theta\right)$ is strictly convex in $\theta$ (see,
e.g. \cite{Huang_01}) and is uniquely minimized at a (pseudo-) true
parameter $\theta_{0}\in\varTheta$. 
\end{enumerate}
Compared to the rate of convergence of the the M-estimator for the
concave extended linear model, only a few studies have explored its
asymptotic distribution and most of them is on the least squares regression
case (e.g. \cite{Newey_97}, \cite{Huang_03}, \cite{BCCK_15}). Recently,
\cite{Shang-Cheng_13} obtain a general result on point-wise asymptotic
normality of the M-estimator (\ref{eq:m-estimator}) based on functional
Bahadur representation. In proving the asymptotic normality, however,
they impose the smoothness condition on $\rho$ so that it should
be three times continuously differentiable with respect to $\theta$.
On the other hand, our new approach does not require the smoothness
of $\rho$. The following is our motivating example.
\begin{example*}[$L_{1}$ regression.]
 Consider a nonparametric regression model with additive errors:
\begin{align}
y & =\left\langle x,\theta\right\rangle +\varepsilon,\label{eq:regression}
\end{align}
where the regressor $x\in\mathscr{H}$ and the error term $\varepsilon$
are mutually independent random variables; $\varepsilon$ is assumed
to be homoskedastic; and the conditional median of $\epsilon$ given
$x$ is zero, i.e., $\inf\left\{ q:\ P_{\varepsilon}\left(q\mid x\right)\geq\frac{1}{2}\right\} =0$
where $P_{\varepsilon}\left(\cdot\mid x\right)$ is the distribution
function of $\varepsilon$ conditional on $x$. We are interested
in estimating $\theta\in\mathscr{H}$. Suppose we have observations
$Z_{i}=\left(y_{i},x_{i}\right):y_{i}\in\mathbb{{R}},x_{i}\in\mathscr{H},i=1,\ldots,n$
independently drawn from the regression model (\ref{eq:regression}).
With them, we may estimate $\theta$ via $L_{1}$ minimization with
roughness penalty(see, for example \cite{Koenker-Ng-Portnoy_94}):
\begin{align}
\hat{\theta}_{n} & =\arg\min_{\theta\in\Theta}\frac{1}{n}\sum_{i=1}^{n}\left|y_{i}-\left\langle x_{i},\theta\right\rangle \right|+\frac{\lambda}{2}\left\Vert \theta\right\Vert ,\label{eq:lad}
\end{align}
where $\lambda$ is the smoothing parameter that converges to zero
as $n\rightarrow\infty$. Obviously, this example gives the case in
which the criterion function $\rho=\left|\cdot\right|$ is not continuously
differentiable.
\end{example*}

Uniform convergence of the objective function in (\ref{eq:m-estimator})
to its population counterpart:
\begin{alignat*}{1}
\sup_{\theta\in\Theta}\left|\frac{1}{n}\sum_{i=1}^{n}\rho\left(\theta,Z_{i}\right)-\mathbb{E}_{Z}\left[\rho\left(\theta,Z\right)\right]\right| & \overset{p}{\rightarrow}0,
\end{alignat*}
guarantees both consistency of of $\hat{\theta}_{n}$ and convergence
of the optimal value of the objective function. In order to make the
objective function satisfy the uniform convergence, we have to impose
some compactness of the parameter space. These assumptions are rather
restrictive for fully nonparametric settings. It is because of the
theorem by Bakhvalov (Theorem 12.1.1. of \cite{Dudley_99}) . When
$\rho=\left|\cdot\right|$ and $\theta$ is in an infinite-dimensional
space, we have
\begin{alignat*}{1}
\sup_{\theta}\left|\frac{1}{n}\sum_{i=1}^{n}\rho\left(\theta,Z_{i}\right)-\mathbb{E}_{Z}\left[\rho\left(\theta,Z\right)\right]\right| & \geq\gamma n^{-1/\infty}.
\end{alignat*}
The left-hand side of the inequality does not converge uniformly.

Since $\rho$ is convex, it seems that we may use the convexity lemma
(e.g., \cite{Pollard_91} and \cite{Kato_09}) to ensure that point-wise
convergence of convex functions implies uniform convergence. In the
infinite-dimensional case, however, this argument for uniform convergence
may fail. Let $\pi_{n},\ n=1,2,\cdots$ be the sequence of projection
operators on $\mathscr{H}$ onto $E_{n}\subset\mathscr{H}$ where
$E_{n}\subsetneqq E_{m>n}$. Consider a quadratic form $\left\langle \pi_{n}\theta,\theta\right\rangle $
for $\forall\theta\in\mathscr{H}$ that is considered as a convex
function of $\theta$. Then, as $n\rightarrow\infty$, $\left\langle \pi_{n}\theta,\theta\right\rangle $
converges point-wise to $\left\langle \theta,\theta\right\rangle $
but not uniformly.

To solve the aforementioned lack-of-uniform-convergence issue, we
shall propose to apply an alternative mode of convergence, \textit{Mosco
convergence}, which is weaker than uniform convergence but still strong
enough to enable statistical applications. Mosco convergence of the
objective function ensures the convergence of its minimizer (\cite{Attouch_84}).
We develop narrow convergence theory with respect to the Mosco metric,
see also \cite{Geyer_94}, \cite{Dupacava-Wets_88}, \cite{Molchanov_05},
\cite{Knight_03} and \cite{Bucher-Segers-Volgushev_14}. There exist
alternative forms of convergence that is equivalent to Mosco convergence
but more easily verifiable. They include graph convergence (G-convergence)
of subdifferential operators and strong convergence of resolvent.
We shall explain these key concepts in Section 2. Using these equivalences,
we can establish the consistency and narrow convergence of an M-estimator
in an infinite-dimensional parameter space. Furthermore, Mosco convergence
also ensures the invertibility of the ``Hessian'' operator. 

If the parameter space is weakly compact, Mosco convergence of the
convex objective function $F_{n}(\theta)$ in M-estimation ensures
that both empirical minimizer $\hat{\theta}_{n}$ and empirical optimal
value function $F_{n}(\hat{\theta}_{n})$ will converges to the true
parameter $\theta_{0}$ and the true optimal value function $F_{0}(\theta_{0})$
respectively. This property makes it possible to derive the asymptotic
distribution of the optimal value function $F_{n}(\hat{\theta}_{n})$.
Namely,
\begin{enumerate}
\item The convex objective function $F_{n}\left(\theta\right)$ is locally
asymptotically normal at $\theta_{0}$ in Mosco topology:
\begin{alignat*}{1}
n\left[F_{n}\left(\theta_{0}+\frac{1}{\sqrt{n}}t\right)-F_{n}\left(\theta_{0}\right)\right] & \rightsquigarrow\left\langle t,W\right\rangle +\frac{1}{2}\left\langle Vt,t\right\rangle ,
\end{alignat*}
where $W$ is a normal random vector in a Hilbert space and $V=\nabla_{\theta}^{2}F_{0}$
is the ``Hessian'' operator that is almost surely invertible. 
\item The asymptotic distribution of the optimal value function $F_{n}(\hat{\theta}_{n})$
is 
\[
n\left[F_{n}(\hat{\theta}_{n})-F_{n}\left(\theta_{0}\right)\right]\rightsquigarrow\left\langle \hat{t},W\right\rangle +\frac{1}{2}\left\langle V\hat{t},\hat{t}\right\rangle ,
\]
where $\hat{t}=\sqrt{n}(\hat{\theta}_{n}-\theta_{0})$.
\end{enumerate}
As a by-product, the asymptotic distribution of the likelihood ratio
statistic can be derived. These results are established in a fully
nonparametric setting.

The rest of this paper is organized as follows. In Section 2, we describe
the Mosco convergence and introduce the narrow convergence in the
Mosco topology. In Section 3, we derive local asymptotic normality
of an convex objective function in an infinite-dimensional Hilbert
space. We also provide the asymptotic distribution of the likelihood
ratio statistic by using the local asymptotic normality. Appendixes
give some technical lemmas.

\subsection*{Notations}

Let $\rightsquigarrow$ denote narrow convergence and $\xrightarrow{P}$
denote convergence in probability. We use empirical process notation:
$\mathbb{G}_{n}\rho=\frac{1}{\sqrt{n}}\sum_{i=1}^{n}\rho\left(\theta,Z_{i}\right)-\mathbb{E}\left[\rho\left(\theta,Z_{i}\right)\right]$.
We denote $\left\Vert \theta\right\Vert $ as $l_{2}$-norm or $L_{2}$-norm
of an element of Hilbert space $\theta\in\mathscr{H}$. Let $\theta_{n}\xrightarrow{s}\theta_{0}$
denote convergence in strong topology, e,g,$\left\Vert \theta_{n}-\theta_{0}\right\Vert \rightarrow0$
and $\theta_{n}\xrightarrow{w}\theta_{0}$ denote convergence in weak
topology, e,g, $\left\langle \theta_{n},\theta^{*}\right\rangle \rightarrow\left\langle \theta_{0},\theta^{*}\right\rangle $
for all identical dual $\theta^{*}\in\mathscr{H}^{*}\left(=\mathscr{H}\right)$.
We denote the limit in weak topology as $\textrm{w-}\lim_{n\rightarrow\infty}\theta_{n}$.

\section{Mosco Convergence}

First, we introduce a mode of convergence, \emph{Mosco convergence},
for proper lower semi-continuous (l.s.c.) convex functions on a real
separable Hilbert space. For l.s.c. convex functions on a finite dimensional
Euclidean space, point-wise convergence is equivalent to locally uniform
convergence. For functions defined on an infinite-dimensional space,
however, this is not the case. Mosco convergence, on the other hand,
still ensures $\arg\min$ convergence of l.s.c. convex functions on
an infinite-dimensional space, though it is weaker than locally uniform
convergence. In this section, we also provide preliminary results
related to Mosco convergence for later use. 

Mosco convergence and similar concepts in a non-stochastic environment
are considered in \cite{Mosco_69}, \cite{Attouch_84} and \cite{Beer_93}.
Mosco convergence is particularly useful in the context of functional
optimization, making it well suited to M-estimation. 
\begin{defn}
{[}Mosco Convergence{]}\\
 Let $f_{n}:\mathscr{{H}\rightarrow\left(-\infty,\infty\right],\ }n=1,2,\dots$
be a sequence of proper l.s.c. convex functions. $f_{n}$ is said
to be Mosco-convergent to the l.s.c. convex function $f:\mathscr{{H}\rightarrow(-\infty,\infty]}$
if and only if the following two conditions hold.\\
(M1) For each $\theta\in\mathscr{H}$, there exist a convergent sequence
$\theta_{n}\overset{s}{\rightarrow}\theta$ such that ${\displaystyle \limsup_{n}f_{n}\left(\theta_{n}\right)\leq f\left(\theta\right)}$.\\
(M2) ${\displaystyle \liminf_{n}f_{n}\left(\theta_{n}\right)}\geq f\left(\theta\right)$
whenever $\theta_{n}\overset{w}{\rightarrow}\theta$. \\
In this paper, we let ``$f_{n}\stackrel{M}{\rightarrow}f$'' denote
``$f_{n}$ Mosco-converges to $f$.'' 
\end{defn}
The variational properties of Mosco convergence are given by the following
theorem (Theorem 1.10 in \cite{Attouch_84}), which ensures the convergence
of both empirical minimizer and empirical minimum value of the objective
function to the true ones. Suppose $\arg\min f_{n}\neq\textrm{Ø}$,
and existence of $\arg\min f_{n}$ and $\inf f_{n}$ are proved in
Appendix A.3.
\begin{thm}
\label{thm:argmin} We assume the same definitions for $f_{1},f_{2},\cdots$
and $f$. If $f_{n}\stackrel{M}{\rightarrow}f$, then 
\begin{alignat*}{1}
\limsup_{n\rightarrow\infty}\left(\arg\min f_{n}\right) & \subset\arg\min f,
\end{alignat*}
in the weak topology, e.g.,
\begin{alignat*}{1}
\left\langle \arg\min f_{n},h\right\rangle  & \rightarrow\left\langle \arg\min f,h\right\rangle \quad\left(\forall h\in\mathscr{H}^{*}\right),
\end{alignat*}
 where the $\limsup$ is defined as\\
\begin{alignat*}{1}
\limsup_{n\rightarrow\infty}F_{n} & \triangleq\left\{ \textrm{w-}\lim_{n\rightarrow\infty}y_{n_{k}}\ :\ y_{n_{k}}\in F_{n_{k}}\ \textrm{for some}\ n_{k}\rightarrow\infty\right\} .
\end{alignat*}
If there is a weakly compact set $K\subset\mathscr{H}$ such that
$\arg\min f_{n}\subset K$ for all $n$ , then $\lim_{n\rightarrow\infty}\left(\inf f_{n}\right)=\inf f.$
\end{thm}

It is difficult to prove Mosco convergence directly in general settings.
Fortunately, several equivalence conditions for Mosco convergence
are known in the literature. One of the most convenient conditions
for Mosco convergence is point-wise convergence of subdifferentials
of functions. 

To deal with this mode of convergence, we introduce several basic
tools in convex analysis: subdifferential and resolvent. For more
details and proofs on these subjects, see \cite{Aubin-Frankowska_90}.
For fixed $Z\in E$, we can define a set-valued mapping $\partial\rho\left(\theta,Z\right):\Theta\times E\rightarrow\mathscr{H}$
by 
\begin{alignat*}{1}
\partial\rho\left(\theta,Z\right) & =\left\{ \theta\in\mathscr{H}:\ {}^{\forall}\zeta\in\mathscr{H},\ \rho\left(\zeta,Z\right)\geq\rho\left(\theta,Z\right)+\left\langle \zeta-\theta,\ \theta\right\rangle \right\} .
\end{alignat*}
Such $\partial\rho\left(\theta,\cdot\right)$ is said to be the \textit{subdifferential}
of $\rho$ at $\theta$. For each fixed $\theta$, $\partial\rho\left(\theta,Z\right)$
is considered as a possibly set-valued function of $Z$. We may regard
$\partial\rho\left(\theta,Z\right)$ as a generalized derivative of
$\rho$ at $\theta$, for each fixed $Z$. If $\rho$ is G\^{a}teaux
differentiable at $\theta$ and has a continuous G\^{a}teaux derivative
$\nabla\rho\left(\theta\right)$, then $\partial\rho\left(\theta,Z\right)=\nabla\rho\left(\theta,Z\right)$. 
\begin{example*}[$L_{1}$ regression(continued)]
 The criterion function $\rho\left(\theta,Z\right)=\left|y-\left\langle x,\theta\right\rangle \right|$
is a proper l.s.c. convex function and has the subdifferential such
that
\begin{alignat*}{1}
\partial\rho\left(\theta,Z\right) & =\begin{cases}
\textrm{sgn}\left(y-\left\langle x,\theta\right\rangle \right)x, & \text{{if}}\ y-\left\langle x,\theta\right\rangle \neq0;\\
\left[-1,1\right]x, & \text{{if}}\ y-\left\langle x,\theta\right\rangle =0,
\end{cases}
\end{alignat*}
where $\textrm{sgn}\left(y-\left\langle x,\theta\right\rangle \right)=\begin{cases}
1, & \text{{if}}\ \left(y-\left\langle x,\theta\right\rangle \right)>0\\
-1, & \text{{if}}\ \left(y-\left\langle x,\theta\right\rangle \right)<0.
\end{cases}$ \end{example*}
\begin{proof}
Proof is given in Appendix A.1.
\end{proof}

\begin{lem}[``Optimization Theory'' Indicator Function]
\label{-The-indicator} The indicator function $\Psi_{A}$ is defined
by
\begin{alignat*}{1}
\Psi_{A}\left(\theta\right) & =\begin{cases}
0 & \left(\theta\in A\right)\\
\infty & \left(\theta\notin A\right)
\end{cases}
\end{alignat*}
where the set $A$ is a convex subset of $\Theta$. The normal cone
$N_{A}\left(a\right)$ is defined by
\begin{alignat*}{1}
N_{A}\left(a\right) & =\left\{ \theta^{\star}\in\mathscr{H}:\ \left\langle \theta-a,\theta^{\star}\right\rangle \leqq0,\ \forall\theta\in A\right\} .
\end{alignat*}
Then, $N_{A}\left(a\right)=\partial\Psi_{A}\left(a\right)$, where
$N_{A}\left(a\right)$ is such that $0\in N_{A}\left(a\right)$.\end{lem}
\begin{proof}
a
\begin{alignat*}{1}
\theta^{\star}\in\partial\Psi_{A}\left(a\right) & \Leftrightarrow\Psi_{A}\left(a\right)+\left\langle \theta-a,\theta^{\star}\right\rangle \leqq\Psi_{A}\left(\theta\right)\ \left(\forall\theta\in A\right)\\
 & \Leftrightarrow\left\langle \theta-a,\theta^{\star}\right\rangle \leqq\Psi_{A}\left(\theta\right)\ \left(\forall\theta\in A\right)\\
 & \Leftrightarrow\left\langle \theta-a,\theta^{\star}\right\rangle \leqq0\ \left(\forall\theta\in A\right)\\
 & \Leftrightarrow\theta^{\star}\in N_{A}\left(a\right)
\end{alignat*}
Then, $N_{A}\left(a\right)=\partial\Psi_{A}\left(a\right)$.
\end{proof}

Subdifferential operator for proper l.s.c. convex funtions holds distributive
law:
\begin{alignat*}{1}
\partial\left(f_{1}+f_{2}\right) & =\partial f_{1}+\partial f_{2}
\end{alignat*}
where $f_{1}$ and $f_{2}$ are proper l.s.c. convex functions on
$\mathscr{H}$ (see Theorem 3.16. in \cite{Phelps_92}). When $\mathscr{H}$
is real separable, subdifferential operator is exchangeable with respect
to integral (\cite{Clarke_83} page 76.):
\begin{alignat*}{2}
\partial f\left(\theta\right) & =\partial\int_{E}f\left(\theta,Z\right)\mathbb{P}_{Z}\left(dZ\right) & =\int_{E}\partial f\left(\theta,Z\right)\mathbb{P}_{Z}\left(dZ\right).
\end{alignat*}

\begin{example*}[$L_{1}$ regression (continued).]
 The limit criterion $\mathbb{E}\left[\left|y-\left\langle x,\theta\right\rangle \right|\right]$
is convex function and has the subdifferential
\begin{alignat*}{1}
\partial\mathbb{E}\left[\left|y-\left\langle x,\theta\right\rangle \right|\right] & =\mathbb{E}\left[\partial\left|y-\left\langle x,\theta\right\rangle \right|\right],
\end{alignat*}
and 
\begin{equation}
\begin{aligned}\mathbb{E}\left[\partial\left|y-\left\langle x,\theta\right\rangle \right|\right] & =\mathbb{E}\left[x\cdot\textrm{sgn}\left(y-\left\langle x,\theta\right\rangle \right)\right]\\
 & =\mathbb{E}\left[\mathbb{E}\left[x\left\{ 1-2\mathbb{I}\left(y-\left\langle x,\theta\right\rangle \leq0\right)\right\} \left|x\right.\right]\right]\\
 & =\mathbb{E}\left[x\left\{ 1-2P_{\varepsilon}\left(q-\left\langle x,\theta\right\rangle \left|x\right.\right)\right\} \right].
\end{aligned}
\label{eq:sub_LAD}
\end{equation}
where $P_{\varepsilon}\left(\cdot\mid x\right)$ is the distribution
function of $\varepsilon$ conditional on $x$.
\end{example*}

In this paper, we assume that the subdifferential $\partial\rho$
is selected and measurable in $Z$. In general, because $\partial\rho$
is a set-valued mapping, the selection is not unique. Nonetheless,
we can show that not only such measurable selections exists but also
the set of all measurable selector $S_{\partial\rho}$ is identical
to $\partial\rho$. 
\begin{prop}
\label{prop:3}There exists a measurable selector of the subdifferential
$\partial f$, i.e., $S_{\partial f}\neq\emptyset$. Moreover, $S_{\partial f}=\partial f$. \end{prop}
\begin{proof}
Proof is given in Appendix A.2.
\end{proof}
Consider a map
\begin{alignat*}{1}
J_{\lambda}^{\partial f}\theta & =\left\{ z\in\mathscr{H}\ :\ z+\lambda\partial f\left(z\right)\ni\theta\right\} .
\end{alignat*}
Such a map should be single-valued (on Proposition 3.5.3 in \cite{Aubin-Frankowska_90}).
Such $J_{\lambda}^{\partial f},\lambda>0$ are called \textit{resolvents}
of $\partial f$ and denoted by 
\begin{alignat*}{1}
^{\forall}\lambda>0,\quad J_{\lambda}^{\partial f} & =\left(I+\lambda\partial f\right)^{-1}.
\end{alignat*}

The following theorem states the equivalence between Mosco convergence
and strong convergence of resolvents and G-convergence of subdifferential
operators. The proofs are given in Theorem 3.26. and Theorem 3.66.
of \cite{Attouch_84}.
\begin{thm}
\label{thm:=00005BAttouch(1984)=00005D} Let $\mathscr{H}$ be a real
separable Hilbert space. Let $\left(f_{n}\right)_{n\in\mathbb{N}}$,
$f_{n}:\ \mathscr{H}\rightarrow\left(-\infty,\infty\right],\ ^{\forall}n\in\mathbb{N}$
be a proper l.s.c. convex function. The following statements are equivalent.
\\
(1) $f_{n}\stackrel{M}{\longrightarrow}f_{0}$.\\
(2) $^{\forall}\lambda>0$, $^{\forall}\theta\in\mathscr{H}$, $J_{\lambda}^{\partial f_{n}}\theta\rightarrow J_{\lambda}^{\partial f}\theta$
strongly in $\mathscr{H}$ as $n$ goes to $\infty$.\\
(3) $\begin{cases}
\partial f_{n}\overset{G}{\rightarrow}\partial f_{0},\\
^{\exists}\left(\theta_{0},\eta_{0}\right)\in\partial f_{0}\ ^{\exists}\left(\theta_{n},\eta_{n}\right)\in\partial f_{n}\ such\ that\ \theta_{n}\overset{s}{\rightarrow}\theta_{0},\ \eta_{n}\overset{s}{\rightarrow}\eta_{0},\ f_{n}\left(\theta_{n}\right)\rightarrow f_{0}\left(\theta_{0}\right),
\end{cases}$\\
where $\partial f_{n}\overset{G}{\rightarrow}\partial f_{0}$ means
that, for every $\left(\theta_{0},\eta_{0}\right)\in\partial f_{0}$,
there exists a sequence $\left(\theta_{n},\eta_{n}\right)\in\partial f_{n}$
such that $\theta_{n}\rightarrow\theta_{0}$ strongly in $\mathscr{H}$,
$\eta_{n}\rightarrow\eta_{0}$ strongly in $\mathscr{H}^{*}\left(=\mathscr{H}\right)$.
\end{thm}

Statement (3) in Theorem \ref{thm:=00005BAttouch(1984)=00005D} is
called G-convergence of monotone operators. This states that point-wise
convergence of all measurable selectors of subdifferential operators
is equivalent to Mosco convergence of functionals. When the subdifferential
is calculable, point-wise convergence of measurable selectors are
easy to verify.
\begin{example*}[$L_{1}$ regression(continued)]
 From the foregoing theorems, it will be seen that the law of large
numbers(LLN) of subdifferential $\partial\rho\left(\theta\right)$
implies the Mosco convergence. From Lemma \ref{lem:9} and the LLN
in Banach spaces for each sequence of mesurable selectors of $\partial\rho\left(\theta\right)$,
we have the LLN of subdirrential $\partial\rho\left(\theta\right)$:\\
\begin{alignat*}{1}
\frac{1}{n}\sum_{i=1}^{n}\partial\rho\left(\theta,Z_{i}\right) & \xrightarrow{P}\mathbb{E}\left[\partial\rho\left(\theta,Z\right)\right]\\
 & =\partial\mathbb{E}\left[\rho\left(\theta,Z\right)\right].
\end{alignat*}
Thus this fact establish the consistency of local functional estimation.
\end{example*}

(2) in the above theorem give a metric that induces the Mosco convergence.
Based on resolvet, \cite{Attouch_84} (p. 365) gives a metric that
induces graph convergence on the space of subdifferential operators:
\begin{alignat*}{1}
d_{G}\left(\partial f,\partial g\right) & \triangleq\sum_{k\in\mathbb{N}}\frac{1}{2^{k}}\inf\left\{ 1,\left\Vert J_{\lambda_{0}}^{\partial f}\theta_{k}-J_{\lambda_{0}}^{\partial g}\theta_{k}\right\Vert \right\} ,
\end{alignat*}
for any subdifferential operators $\partial f$ and $\partial g$
where $\lambda_{0}$ is taken strictly positive and $\left\{ \theta_{k};k\in\mathbb{N}\right\} $
is a dense subset of $\mathscr{H}$. This metric $d_{G}$ induces
the Mosco convergence topology and is complete. Convergence in $d_{G}$
are equivalent to the convergence results in (1)$\sim$(3) in Theorem
\ref{thm:=00005BAttouch(1984)=00005D}.

Hoffman-Jørgensen weak convergence theory performs in a metric space.
Generally, epi-convergence does not usually work with a metric but
a semi-metric. Even if functions $f,g$ are different each other,
it is possible $f$ epi-converge to $g$ (see, Section 3 in \cite{Bucher-Segers-Volgushev_14}).
Fortunately in the case where the functional space is constituted
by convex functions, we can obtain a metric space as described above.
We shall define a weak convergence in the following way. 
\begin{defn}
{[}Mosco Convergence in Distribution{]}\\
A sequence of random elements $f_{n}$ in the space of proper l.s.c.
convex functions $\mathscr{H}\rightarrow\left(-\infty,\infty\right]$
is said to be Mosco converges in distribution to the random element
$f_{0}$ in the space of proper l.s.c. convex functions if $f_{n}\rightsquigarrow f_{0}$
with metric $d_{G}$. We use the notation $f_{n}\stackrel{M}{\rightsquigarrow}f_{0}$.
\end{defn}

\section{Local Asymptotic Normality}

First, we show that the reparametrized objective function admits a
certain quadratic expansion. A common starting point in developing
an asymptotic distribution theory for an M-estimator is to define
a centered stochastic process based on the objective function. Recall
that $F_{n}(\theta)=\frac{1}{n}\sum_{i}\rho\left(\theta,Z_{i}\right)$
is the objective function for the M-estimator (\ref{eq:m-estimator}).
We may define such a centered stochastic process as
\begin{alignat}{1}
H_{n}\left(\theta,t\right) & \triangleq n\left[F_{n}\left(\theta+\frac{1}{\sqrt{n}}t\right)-F_{n}\left(\theta\right)\right],\label{eq:H_n}
\end{alignat}
\begin{alignat}{1}
Q_{0}\left(t\right) & \triangleq\left\langle t,W\right\rangle +\frac{1}{2}\left\langle Vt,t\right\rangle ,\label{eq:Q_0}
\end{alignat}
where $W$ is an $N\left(\boldsymbol{0},A\right)$ random vector in
a Hilbert space and $V$ is a ``Hessian'' operator. $H_{n}\left(\theta_{0},t\right)$
is interpreted as the log likelihood ratio for hypothesis testing
against the local alternative, i.e., $\mathcal{{H}}_{0}:\ \theta=\theta_{0};\ \mathcal{{H}}_{1}:\ \theta=\theta_{0}+\frac{1}{\sqrt{{n}}}t$.
Define the locally asymptotically quadratic (LAQ) as follows.
\begin{defn}[LAQ \cite{LeCum-Yang_00}(p. 120))]
 The convex objective function $F_{n}\left(\theta\right)$ is said
to be locally asymptotically quadratic at $\theta$ if there exists
a random matrix $V_{n,\theta}$ and a random vector $\Delta_{n,\theta}$
such that 
\begin{alignat*}{1}
H_{n}\left(\theta,t\right) & =\left\langle t,\Delta_{n,\theta}\right\rangle +\frac{1}{2}\left\langle V_{n,\theta}t,t\right\rangle +o_{p_{n,\theta}}\left(1\right),
\end{alignat*}
and the matrix $V_{n,\theta}$ and their limit $\left(V_{n,\theta}\rightsquigarrow\right)V_{\theta}$
are almost surely invertible. 
\end{defn}

\begin{rem*}
Recall that locally asymptotically mixed normality (LAMN) is equivalent
to LAQ with a restriction: $\Delta_{n,\theta},V_{n,\theta}$ converge
to normal distributions. Locally asymptotically normality (LAN) is
equivalent to LAMN with the limiting matrix $V_{\theta}$ is deterministic. 
\end{rem*}

\subsection{Second Order Differentiability}

In typical situations, we assume that the function $F_{0}$ has a
quadratic expansion at $\theta_{0}$ and their Hessian is often supposed
to be continuously invertible (Theorem 3.3.1. of \cite{vdVaart-Wellner_96}).
In an infinite-dimensional case, the assumption that the Hessian operator
is continuously invertible is harder to ascertain. However, if the
convex function $F_{0}$ has a\emph{ }\textit{generalized second order
differentiability} (defined later), its ``generalized Hessian''
is continuously invertible. 

Define the Young-Fenchel conjugate $f^{*}$ of convex function $f$
as
\begin{alignat*}{1}
f^{*}\left(\eta\right) & \triangleq\sup_{\theta}\left(\left\langle \eta,\theta\right\rangle -f\left(\theta\right)\right).
\end{alignat*}
The conjugate $f^{*}$ has a strong link between a convex function
$f$ in the second order differentiability. Recall the case of a convex
function defined on finite dimensional parameters. A convex function
$f$ defined on the Euclid space $\mathbb{R}^{d}$ is second order
differentiable and the Hessian $\nabla^{2}f\left(\theta\right)$ of
$f$ at $\theta$ is nondegererate. Then the conjugate function $f^{*}$
is second order differentiable at $y=\nabla f\left(\theta\right)$,
and its Hessian $\nabla^{2}f^{*}\left(\eta\right)$ at $y$ is the
inverse of $\nabla^{2}f\left(\theta\right)$, i.e.,
\begin{alignat*}{1}
\nabla^{2}f\left(\theta\right) & =\left(\nabla^{2}f^{*}\left(\eta\right)\right)^{-1}.
\end{alignat*}
In order to maintain a duality-type of this relation in an infinite-dimensional
space, we shall define the second order differential concepts based
on Mosco convergence. Mosco convergence ensures the continuity of
this type of conjugation (\cite{Kato_89} and \cite{Borwein-Noll_94}). 

Define \textit{second difference quotient} of $f$ at $\theta\in\mathscr{H}$
relative to $\eta^{*}\in\partial f\left(\theta\right)$ as
\begin{alignat*}{1}
\Delta_{f,\theta,\eta,t}\left(h\right) & \triangleq\frac{f\left(\theta+th\right)-f\left(\theta\right)-t\left\langle \eta^{\star},h\right\rangle }{t^{2}}
\end{alignat*}
and define a \textit{purely quadratic} continuous convex function
as
\begin{alignat*}{1}
q\left(h\right) & \triangleq\frac{1}{2}\left\langle Vh,h\right\rangle ,
\end{alignat*}
where $V$ is a closed symmetric positive linear operator. $f$ is
said to have \textit{generalized second order differentiability}\textbf{\emph{
}}at $\theta$ relative to $\eta^{\star}\in\partial f\left(\theta\right)$
if there exists a purely quadratic function $q$ such that the second
order difference quotient $\Delta_{f,\theta,\eta,t}\left(\cdot\right)$
converges to $q\left(\cdot\right)$ in the Mosco sense, i.e.,
\begin{alignat*}{1}
\Delta_{f,\theta,\eta,t}\left(h\right) & \stackrel[t\downarrow0]{M}{\longrightarrow}q\left(h\right).
\end{alignat*}
The closed symmetric positive linear operator $V$ is called the \textit{generalized
Hessian} of $f$ at $\theta$ relative to $\eta\in\partial f\left(\theta\right)$.

Mosco convergence is invariant under Young-Fenchel conjugation, so
that Mosco convergence of $\Delta_{f,\theta,\eta,t}\left(h\right)$
is equivalent to Mosco convergence of $\left(\Delta_{f,\theta,\eta,t}\left(h\right)\right)^{*}=\Delta_{f*,\eta,\theta,t}\left(h\right)$.
And generalized Hessian of $f^{*}$ at $\eta$ relative to $\theta\in\partial f^{*}\left(\eta\right)$
is $V^{-1}$.

Next, we derive sufficient conditions under which the objective function
of M-estimation has generalized second order differentiability. $\partial f$
is called\textbf{ }\textit{weak{*} G\^{a}teaux differentiable} at
\textbf{$\theta$ }if there exists a bounded linear operator $T:\mathscr{H}\rightarrow\mathscr{H}^{*}$
such that
\begin{alignat*}{1}
\lim_{t\rightarrow0}\frac{1}{t}\left(\eta_{t}^{*}-\eta^{*}\right) & =Vh,
\end{alignat*}
in the weak{*} sense for any fixed $h\in\mathscr{H}$ and all $\eta_{t}^{*}\in\partial f\left(\theta+th\right)$,
$\eta^{*}\in\partial f\left(\theta\right)$ where $\partial f\left(\theta\right)$
must consist of a single element $\eta^{*}$ . We use the notation
$T=\nabla\partial f\left(\theta\right)$ for the operator $T$. For
the generalized differentiability, we quote the following result of
\cite{Borwein-Noll_94}.
\begin{thm}
(a variant of Propotion 6.4. of \cite{Borwein-Noll_94}) \\
Let $\left(Z,\mathcal{Z},\mathbb{P}_{Z}\right)$ be a probability
space and $\Theta\subseteq\mathscr{H}$ be a separable Hilbert space.
Suppose $\rho:\Theta\times Z\rightarrow\left(-\infty,\infty\right]$
is measurable on $\left(Z,\mathcal{{Z}},\mathbb{{P}}_{Z}\right)$
and convex at any $\theta\in\Theta$ and define a closed convex integral
functional $f$ on $\Theta\subset\mathscr{H}$ as
\begin{alignat*}{1}
f\left(\theta\right) & =\int_{Z}\rho\left(\theta,z\right)d\mathbb{P}_{Z}\left(z\right).
\end{alignat*}
Then $f$ is generalized second order differentiable at $\theta$
if and only if $\partial\rho$ is weak{*} G\^{a}teaux differentiable
and 
\begin{alignat*}{1}
\mathrm{{ess}}\sup_{z\in Z}\left|\nabla\partial\rho\left(\theta,z\right)\right|<\infty & .
\end{alignat*}

\end{thm}

\begin{example*}[$L_{1}$ regression(continued)]
 Let $Z=\left(Y,X\right)$ be a random vector, where $Y$ is real-valued
while $X$ is the covariate and $X\in\mathscr{H}$. Note that objective
function of $L_{1}$ regression is
\begin{alignat*}{1}
F\left(\theta\right) & =\mathbb{E}\left[\left|Y-\left\langle x,\theta\right\rangle \right|\right]\\
 & =\mathbb{E}\left[\mathbb{E}\left[\left|Y-\left\langle x,\theta\right\rangle \right|\left|X\right.\right]\right].
\end{alignat*}
Then, $L_{1}$ regression objective function $F\left(\theta\right)$
is generalized second order differentialbe at $\theta$ if and only
if $\partial\mathbb{E}\left[\left|Y-\left\langle x,\theta\right\rangle \right|\left|X\right.\right]$
is weak{*} \^{G}ateaux differentiable and 
\begin{alignat*}{1}
\text{\ensuremath{\mathrm{{ess}}}}\sup_{x\in X}\left|\nabla\partial\mathbb{E}\left[\left|Y-\left\langle x,\theta\right\rangle \right|\left|X\right.\right]\right|<\infty & .
\end{alignat*}
From (\ref{eq:sub_LAD}), weak{*} G\^{a}teaux differentiability of
$\partial\mathbb{E}\left[\left|Y-\left\langle x,\theta\right\rangle \right|\left|X\right.\right]$
at $\theta$ is equivalent to the G\^{a}teaux differentiability of
the distribution function $F_{e}\left(q-\left\langle x,\theta\right\rangle \left|x\right.\right)$
at $\theta$. If the distribution function $F_{e}\left(q-\left\langle x,\theta\right\rangle \left|x\right.\right)$
is G\^{a}teaux differentiable at $\theta$, essential boundedness
of $\textrm{ess}\sup_{x\in X}\left|\nabla\partial\mathbb{E}\left[\left|Y-\left\langle x,\theta\right\rangle \right|\left|X\right.\right]\right|<\infty$
will be automatically satisfied.
\end{example*}

Therefore, in order to obtain invertiblity of ``generalized Hessian'',
we impose the following assumption on $\rho$:
\begin{assumption*}
A\\
$\partial\rho\left(\cdot\right)$ is weak{*} G\^{a}teaux differentiable
at $\theta_{0}$ and 
\begin{alignat*}{1}
\text{\ensuremath{\mathrm{{ess}}}}\sup_{z\in E}\left|\nabla\partial\rho\left(\theta_{0},z\right)\right|<\infty & .
\end{alignat*}

\end{assumption*}

This assumption is a ``low-level'' condition which are sufficient
for locally asymptotically quadratic at $\theta_{0}$ than that of
\cite{Geyer_94}. Of course, this result is attributed to the convexity
of the objective function.

\subsection{LAN}

Define auxiliary stochastic process as
\begin{alignat*}{1}
G_{n}\left(t\right) & \triangleq n\left\langle \frac{1}{\sqrt{n}}t,\partial F_{n}\left(\theta_{0}\right)\right\rangle +n\left[F_{0}\left(\theta_{0}+\frac{1}{\sqrt{n}}t\right)-F_{0}\left(\theta_{0}\right)\right],\\
G_{n}^{\prime}\left(t\right) & \triangleq n\left\langle \frac{1}{\sqrt{n}}t,\partial F_{n}\left(\theta_{0}\right)\right\rangle +\frac{1}{2}\left\langle Vt,t\right\rangle .
\end{alignat*}

We also impose the following assumption. Considering Proposition \ref{prop:3}
: the set of all mesurable selectors of a subdifferential coincides
with its own subdifferential, we denote any measurable selector of
$\partial\rho\left(\cdot\right)$ as itself.
\begin{assumption*}
B\\
Every mesurable selector in $\partial\rho\left(\theta,Z\right)$ has
a bounded variance: $\forall\theta\in\Theta$, $\mathbb{E}\left[\left\Vert \partial\rho\left(\theta,Z\right)\right\Vert ^{2}\right]<\infty$,
and there is a sequence of mesurable selectors satisfying a central
limit theorem in the Hilbert space:
\begin{alignat*}{1}
\mathbb{G}_{n}\partial\rho\left(\theta_{0},Z\right) & \rightsquigarrow N\left(0,A\right),
\end{alignat*}
for some trace class covariance operator $A$.
\end{assumption*}

\begin{prop}
\label{prop:LAN}LAN\end{prop}
\begin{enumerate}
\item $H_{n}\left(t\right)$ Mosco-converges to $G_{n}^{\prime}\left(t\right)$
in probability.
\item $G_{n}^{\prime}\left(t\right)$ converges in law to $Q_{0}\left(t\right)$.
Then, $H_{n}\left(t\right)$ Mosco-converge in law to $Q_{0}\left(t\right)$.
\end{enumerate}

\begin{proof}
We shall prove the first statement. In order that $H_{n}\left(t\right)$
converges in Mosco to $G_{n}\left(t\right)$, we will apply Theorem
\ref{thm:=00005BAttouch(1984)=00005D} to $H_{n}\left(t\right)$ and
$G_{n}\left(t\right).$ All we have to do is to show the graph convergence
of the subdifferential $\partial H_{n}\left(t\right)$ to $\partial G_{n}\left(t\right)$
in probability. Considering proposition \ref{prop:3}, we denote any
measurable selector of $\partial\rho\left(\cdot\right)$ as itself
in the following proof below. Calculate subdifferential of $H_{n},G_{n}$
with respect to $t$, we obtain
\begin{alignat*}{1}
\partial H_{n}\left(t\right) & =\sqrt{n}\partial F_{n}\left(\theta_{0}+\frac{1}{\sqrt{n}}t\right)\\
 & =\frac{1}{\sqrt{n}}\sum_{i=1}^{n}\partial\rho\left(\theta_{0}+\frac{1}{\sqrt{n}}t\right),\\
\partial G_{n}\left(t\right) & =\sqrt{n}\partial F_{n}\left(\theta_{0}\right)+\sqrt{n}\partial F_{0}\left(\theta_{0}+\frac{1}{\sqrt{n}}t\right)\\
 & =\frac{1}{\sqrt{n}}\sum_{i=1}^{n}\partial\rho\left(\theta_{0},Z_{i}\right)+\sqrt{n}\mathbb{E}\left[\partial\rho\left(\theta_{0}+\frac{1}{\sqrt{n}}t,Z\right)\right].
\end{alignat*}
Recall $\partial f_{n}\overset{G}{\rightarrow}\partial f_{0}$ means
that for every $\left(\theta_{0},\eta_{0}\right)\in\partial f_{0}$,
there exists a sequence $\left(\theta_{n},\eta_{n}\right)\in\partial f_{n}$
such that $\theta_{n}\rightarrow\theta_{0}$ strongly in $\mathscr{H}$,
$\eta_{n}\rightarrow\eta_{0}$ strongly in $\mathscr{H}^{*}\left(=\mathscr{H}\right)$.
$\partial H_{n}\overset{G}{\rightarrow}\partial G_{n}$ means that
there exists a sequence of measurable selectors of $\frac{1}{\sqrt{n}}\sum_{i=1}^{n}\partial\rho\left(\theta_{0}+\frac{1}{\sqrt{n}}t,Z_{i}\right)$
such that 
\begin{alignat*}{1}
\frac{1}{\sqrt{n}}\sum_{i=1}^{n}\partial\rho\left(\theta_{0}+\frac{1}{\sqrt{n}}t,Z_{i}\right) & \rightarrow\frac{1}{\sqrt{n}}\sum_{i=1}^{n}\partial\rho\left(\theta_{0},Z_{i}\right)+\sqrt{n}\mathbb{E}\left[\partial\rho\left(\theta_{0}+\frac{1}{\sqrt{n}}t,Z\right)\right],
\end{alignat*}
strongly in $\mathscr{H}$. 

The random variable
\begin{alignat*}{1}
 & \partial\rho\left(\theta_{0}+\frac{1}{\sqrt{n}}t,Z_{i}\right)-\partial\rho\left(\theta,Z_{i}\right),
\end{alignat*}
converges monotonically to non-negative random variable. Because $F_{0}\left(\theta\right)=\mathbb{E}\left[\rho\left(\theta\right)\right]$
is second order differentiable in the generalized sense, 
\begin{alignat*}{1}
\mathbb{E}\left[\lim_{n\rightarrow\infty}\partial\rho\left(\theta_{0}+\frac{1}{\sqrt{n}}t,Z_{i}\right)-\partial\rho\left(\theta_{0},Z_{i}\right)\right] & =0,
\end{alignat*}
so,
\begin{alignat*}{1}
\lim_{n\rightarrow\infty}\partial\rho\left(x\theta+\frac{1}{\sqrt{n}}t,Z_{i}\right)-\partial\rho\left(\theta_{0},Z_{i}\right)=0 & \ \textrm{a.s.}.
\end{alignat*}

Fix $t$ and define a (selected) random variable $\xi_{ni}$ by
\begin{alignat*}{1}
\xi_{ni} & =\frac{1}{\sqrt{n}}\partial\rho\left(\theta_{0}+\frac{1}{\sqrt{n}}t,Z_{i}\right)-\frac{1}{\sqrt{n}}\partial\rho\left(\theta_{0},Z_{i}\right).
\end{alignat*}
 Note that
\begin{alignat*}{1}
\mathbb{E}\left[\frac{1}{\sqrt{n}}\partial\rho\left(\theta_{0}+\frac{1}{\sqrt{n}}t,Z\right)-\frac{1}{\sqrt{n}}\partial\rho\left(\theta_{0},Z\right)\right] & =\mathbb{E}\left[\frac{1}{\sqrt{n}}\partial\rho\left(\theta_{0}+\frac{1}{\sqrt{n}}t,Z\right)\right],
\end{alignat*}
where $\mathbb{E}\left[\frac{1}{\sqrt{n}}\partial\rho\left(\theta_{0}+\frac{1}{\sqrt{n}}t,Z\right)\right]$
is singleton. Therefore, for any selected $\xi_{ni}$,
\begin{alignat*}{1}
\sum_{i=1}^{n}\xi_{ni} & =\partial H_{n}\left(t\right)-\partial G_{n}\left(t\right)+\sqrt{n}\mathbb{E}\left[\partial\rho\left(\theta_{0}+\frac{1}{\sqrt{n}}t,Z\right)\right],
\end{alignat*}
and
\begin{alignat*}{1}
\textrm{Var}\left[\sum_{i=1}^{n}\xi_{ni}\right] & =\mathbb{E}\left[\left(\partial H_{n}\left(t\right)-\partial G_{n}\left(t\right)\right)^{2}\right].
\end{alignat*}
Since $\xi_{n1},\dots,\xi_{nn}$ are i.i.d., we have 
\begin{alignat*}{1}
\textrm{Var}\left[\sum_{i=1}^{n}\xi_{ni}\right] & \leq\sum_{i=1}^{n}\mathbb{E}\left[\xi_{ni}^{2}\right],
\end{alignat*}
for any selected $\xi_{ni}$. We have the equality
\begin{alignat*}{1}
\sum_{i=1}^{n}\mathbb{E}\left[\xi_{ni}^{2}\right] & =n\mathbb{E}\left[\left\{ \frac{1}{\sqrt{n}}\partial\rho\left(\theta_{0}+\frac{1}{\sqrt{n}}t,Z_{i}\right)-\frac{1}{\sqrt{n}}\partial\rho\left(\theta_{0},Z_{i}\right)\right\} ^{2}\right]\\
 & =\mathbb{E}\left[\left\{ \partial\rho\left(\theta_{0}+\frac{1}{\sqrt{n}}t,Z_{i}\right)-\partial\rho\left(\theta_{0},Z_{i}\right)\right\} ^{2}\right].
\end{alignat*}
By weak{*} differentiability of $\mathbb{E}\left[\partial\rho\right]$
at $\theta_{0}$, the limit of any measurable selector of $\partial\rho\left(\theta_{0}+\frac{1}{\sqrt{n}}t,Z_{i}\right)-\partial\rho\left(\theta_{0},Z_{i}\right)$
has expectation zero. From the Assumption B : for every measurable
selector $\mathbb{E}\left[\left\{ \partial\rho\left(\theta,Z_{i}\right)\right\} ^{2}\right]<\infty$
for each $\theta$ in the neighborhood of $\theta_{0}$ and from Lebesgue
dominated convergence theorem, we have 
\begin{alignat*}{1}
\mathbb{E}\left[\left\{ \partial\rho\left(\theta_{0}+\frac{1}{\sqrt{n}}t,Z_{i}\right)-\partial\rho\left(\theta_{0},Z_{i}\right)\right\} ^{2}\right] & \rightarrow0,\qquad\left(n\rightarrow\infty\right).
\end{alignat*}
Thus, ${\displaystyle \textrm{Var}\left[\sum_{i=1}^{n}\xi_{ni}\right]\leq\sum_{i=1}^{n}\mathbb{E}\left[\xi_{ni}^{2}\right]\rightarrow0}$.
By Chebyshev inequality, we have
\begin{alignat*}{1}
\partial H_{n}\left(t\right)-\partial G_{n}\left(t\right) & \stackrel{P}{\rightarrow}0,
\end{alignat*}
for fixed $t$. Then, $H_{n}\left(t\right)$ converges in Mosco to
$G_{n}\left(t\right)$ in probability.

From Assumption A, $F_{0}$ is second order differentiable in generalized
sense:
\begin{alignat*}{1}
\frac{\left\{ F_{0}\left(\theta_{0}+\frac{1}{\sqrt{n}}t\right)-F_{0}\left(\theta_{0}\right)-\frac{1}{\sqrt{n}}\left\langle \partial F_{0}\left(\theta_{0}\right),t\right\rangle \right\} }{\left(\frac{1}{\sqrt{n}}\right)^{2}} & \stackrel{M}{\longrightarrow}\frac{1}{2}\left\langle Vt,t\right\rangle .
\end{alignat*}
Therefore, combining aforementioned result, we obtain the result that
$H_{n}\left(t\right)$ Mosco-converges to $G_{n}^{\prime}\left(t\right)$
in probability.

The second statement of Proposition \ref{prop:LAN} is derived from
Assumption B and a.s. representation theorem (Theorem 1.10.4. of \cite{vdVaart-Wellner_96}).
We get
\begin{alignat*}{1}
\xi_{n}\left(=\frac{1}{\sqrt{n}}\sum_{i=1}^{n}\partial\rho\left(\theta_{0},X_{i}\right)\right) & \rightsquigarrow\xi,
\end{alignat*}
and an almost sure representation $\tilde{\xi}_{n}\rightarrow\tilde{\xi}\ a.s.$,
where $\tilde{\xi}_{n}$ has the same law as $\xi_{n}$ and $\tilde{\xi}$
the same law as $\xi$. This provide the Mosco convergence in distribution
of $G_{n}^{\prime}$ to $Q_{0}$.
\end{proof}

The aforementioned proposition achieves mosco convergence of $H_{n}$
to its limit $Q_{0}$. Note that $t=\sqrt{n}\left(\theta-\theta_{0}\right)$
minimizes $H_{n}\left(t\right)$.

Next, we will also show convergence of the minimizer of $H_{n}$ to
that of $Q_{0}$, provided that the minimizer is almost surely unique.
This follows from the following lemma.
\begin{lem}
\label{lem:9}The minimizer of the function $Q_{0}\left(t\right)=\left\langle t,W\right\rangle +\frac{1}{2}\left\langle Vt,t\right\rangle $
is single valued.\end{lem}
\begin{proof}
Let $t_{0}=\arg\min_{t}Q_{0}\left(t\right)$. Suppose there exists
$t_{1}\left(\neq t_{0}\right)$ such that
\begin{alignat*}{1}
\left\langle t_{1},W\right\rangle +\frac{1}{2}\left\langle Vt_{1},t_{1}\right\rangle  & =\left\langle t_{0},W\right\rangle +\frac{1}{2}\left\langle Vt_{0},t_{0}\right\rangle =\alpha.
\end{alignat*}
Then,
\begin{alignat*}{1}
 & \left\langle \frac{t_{1}+t_{0}}{2},W\right\rangle +\frac{1}{2}\left\langle V\frac{t_{1}+t_{0}}{2},\frac{t_{1}+t_{0}}{2}\right\rangle \\
< & \frac{1}{2}\left\langle t_{1},W\right\rangle +\frac{1}{2}\left\langle t_{0},W\right\rangle +\frac{1}{2}\left(\frac{1}{2}\left\langle Vt_{1},t_{1}\right\rangle +\frac{1}{2}\left\langle Vt_{0},t_{0}\right\rangle \right)\\
= & \frac{1}{2}\alpha+\frac{1}{2}\alpha=\alpha.
\end{alignat*}
This means $Q_{0}\left(\frac{t_{1}+t_{0}}{2}\right)<\alpha$, which
is contradiction.
\end{proof}

We apply the previous results to consider the asymptotic distribution
of $\sqrt{n}\left\langle \hat{\theta}-\theta_{0},\theta^{*}\right\rangle $
in the weak topology. 
\begin{cor}
\label{cor:ANormal}Asymptotic Normality\\
Let $W$ be an $N\left(0,A\right)$ distribution. Under Assumption
A and B, we obtain the asymptotic distribution of $\sqrt{n}\left\langle \hat{\theta}_{n}-\theta_{0},\theta^{*}\right\rangle $
as following;
\begin{alignat*}{1}
\sqrt{n}\left\langle \hat{\theta}_{n}-\theta_{0},\theta^{*}\right\rangle  & \rightsquigarrow\left\langle V^{-1}W,\theta^{*}\right\rangle \qquad\forall\theta^{*}\in\Theta
\end{alignat*}
where $V^{-1}$ is generalized Hessian of Young-Fenchel conjugate
of $F_{0}\left(\theta\right)$.\end{cor}
\begin{proof}
From Proposition \ref{prop:LAN}, $H_{n}\left(\theta_{0},\hat{t}\right)$
converges weakly to $Q_{0}\left(t\right)$ in Mosco topology. Applying
a.s. representation theorem(Theorem1.10.4 in \cite{vdVaart-Wellner_96})
we get an almost sure representation $H_{n}\stackrel{M}{\longrightarrow}Q_{0}\ a.s.$.
By Theorem \ref{thm:=00005BAttouch(1984)=00005D} we have
\begin{alignat*}{1}
\lim_{n\rightarrow\infty}\left(\arg\min H_{n}\right) & \rightarrow\arg\min Q_{N}\ a.s.
\end{alignat*}
in the weak topology. This provide 
\begin{alignat*}{1}
\sqrt{n}\left\langle \hat{\theta}_{n}-\theta_{0},\theta^{*}\right\rangle  & \rightsquigarrow\left\langle V^{-1}W,\theta^{*}\right\rangle \qquad\forall\theta^{*}\in\Theta.
\end{alignat*}

\end{proof}

\begin{example*}[$L_{1}$ regression (continued).]
 Suppose the distribution function $F_{e}\left(q-\left\langle x,\theta\right\rangle \left|x\right.\right)$
is G\^{a}teaux differentiable at $\theta$ and denote their differential
as operator $V$. Under Assumption A and B, for any $x_{0}\in\mathscr{H}$,
\begin{alignat*}{1}
\sqrt{n}\left\langle x_{0},\hat{\theta}_{n}-\theta_{0}\right\rangle  & \rightsquigarrow N\left(0,V^{-1}A\right).
\end{alignat*}

\end{example*}

For the implement, we need a consistent estimators of the generalized
Hessian. From the fact of the properties of the generalized differential,
the natural candidates are

\begin{alignat*}{1}
\lim_{h_{n}\rightarrow0}\frac{1}{k_{n}}\left(\hat{\eta}_{k_{n}}^{\star}-\hat{\eta}^{\star}\right)
\end{alignat*}
in the weak{*} sense for any fixed $h\in\mathscr{H}$ and all $\hat{\eta}_{k_{n}}^{\star}\in\partial f\left(\hat{\theta}+k_{n}h\right)$,
$\hat{\eta}^{\star}\in\partial f\left(\hat{\theta}\right)$ .

\subsection{Likelihood Ratio Test Statistic}

Using the previous LAN result, we derives the asymptotic distribution
of the likelihood ratio statistic. Let $A_{n}=\sqrt{n}\left(\Theta-\theta_{0}\right)$
and $A_{n,0}=\sqrt{n}\left(\Theta_{0}-\theta_{0}\right)$. The likelihood
ratio statistic is written by the form
\begin{alignat*}{1}
\Lambda_{n} & =\inf_{t\in A_{n}}H_{n}\left(\theta_{0},t\right)-\inf_{t\in A_{n,0}}H\left(\theta_{0},t\right).
\end{alignat*}
By the previous LAN result, for large $n$, the likelihood ratio process
is similar to the same as in the normal experiment. And by the Mosco
convergence argument in theorem \ref{thm:=00005BAttouch(1984)=00005D},
if the parameter space is weakly compact, the empirical optimal value
of convex function achieve the true optimal.
\begin{assumption*}
C\\
The parameter set $\Theta$ is weakly compact. In a Hilbert space
setting $\Theta\subset\mathscr{H}$, weakly compactness is equal to
boundedness: for all $\theta\in\Theta$,there exists constant $C$
such that $\left\Vert \theta\right\Vert \leq C$.\end{assumption*}
\begin{lem}
\label{lem:LLR}Let $W$ be an $N\left(0,A\right)$ distribution and
repeat (\ref{eq:H_n}); 
\begin{alignat*}{1}
H_{n}\left(\theta,t\right) & =n\left[F_{n}\left(\theta+\frac{1}{\sqrt{n}}t\right)-F_{n}\left(\theta\right)\right].
\end{alignat*}
Let $\hat{t}=\sqrt{n}\left(\hat{\theta}_{n}-\theta_{0}\right)$ denote
this minimizer. Under Assumption A-C, the asymptotic distribution
of the optimal value function
\begin{alignat*}{1}
H_{n}\left(\theta_{0},\hat{t}\right) & =n\left[F_{n}\left(\hat{\theta}_{n}\right)-F_{n}\left(\theta_{0}\right)\right]
\end{alignat*}
is the distribution of $Q_{N}\left(\hat{t}\right)$. \end{lem}
\begin{proof}
From Proposition \ref{prop:LAN}, $H_{n}\left(\theta_{0},\hat{t}\right)$
converges weakly to $Q_{N}\left(t\right)$ in Mosco topology. Applying
a.s. representation theorem(Theorem1.10.4 in \cite{vdVaart-Wellner_96})
we get an almost sure representation $H_{n}\stackrel{Mosco}{\longrightarrow}Q_{N}\ a.s.$.
By Theorem \ref{thm:argmin} and Assumption C, we have
\begin{alignat*}{1}
\lim_{n\rightarrow\infty}\left(\inf H_{n}\right) & =\inf Q_{N}.
\end{alignat*}
 This provide the optimal value of function $H_{n}$ converges weakly
to $Q_{N}$.
\end{proof}

Define an objective function with convex constraint $G\left(\theta\right)$
from $\mathscr{H}$ to $\left(-\infty,\infty\right]$ by
\begin{alignat*}{1}
G_{n}\left(\theta\right) & =F_{n}\left(\theta\right)+\Psi_{A}\left(\theta\right)
\end{alignat*}
where $\Psi_{A}$ is defined by
\begin{alignat*}{1}
\Psi_{A}\left(\theta\right) & =\begin{cases}
0 & \left(\theta\in A\right)\\
\infty & \left(\theta\notin A\right)
\end{cases}
\end{alignat*}
and $A$ is convex. Because $F_{n}$ and $\Psi_{A}$ are convex function,
$G_{n}\left(\theta\right)$ are also convex function with respect
to $\theta$ for all $n$. Redefine (\ref{eq:H_n}), (\ref{eq:Q_0})
as
\begin{alignat*}{1}
H_{n}^{A_{n,0}}\left(\theta,t\right) & \triangleq n\left[F_{n}\left(\theta+\frac{1}{\sqrt{n}}t\right)-F_{n}\left(\theta\right)\right]+\Psi_{A_{n,0}}\left(t\right)\\
Q_{0}^{A}\left(t\right) & \triangleq\left\langle t,Z\right\rangle +\frac{1}{2}\left\langle Vt,t\right\rangle +\Psi_{T_{A_{0}}\left(\theta_{0}\right)}\left(t\right)
\end{alignat*}
where $T_{A}\left(\theta\right)$ is tangent cone:
\begin{alignat*}{1}
T_{A_{0}}\left(\theta\right) & =\limsup_{\tau\downarrow0}\frac{\Theta_{0}-\theta_{0}}{\tau}.
\end{alignat*}
From the result of lemma (\ref{-The-indicator}) and lemma (\ref{lem:LLR}),
we obtain the asymptotic distribution of the optimal value function
\begin{alignat*}{1}
H_{n}^{A}\left(\theta_{0},\hat{t}\right) & \rightsquigarrow Q_{N}\left(\hat{t}\right).
\end{alignat*}
The above result yeilds the asymptotic distribution of the likelihood
ratio statistics $\Lambda_{n}$. The proof strategy is based on \cite{vdVaart_98},
Chapter 16, Theorem 16.7.
\begin{prop}
\label{prop:likeRatio}Assume the parameter spaces $\Theta$ and $\Theta_{0}$
is convex. And assume Assumption A-C. If the sets $A_{n}$ and $A_{n,0}$
converge to sets $A$ and $A_{0}$, then the sequence of likelihood
ratio statistics $\Lambda_{n}$ converges under $\theta_{0}+\frac{t}{\sqrt{n}}$
in distribution to 
\begin{alignat*}{1}
 & \left\Vert V^{-\frac{1}{2}}W+V^{\frac{1}{2}}t\left(\in A_{n,0}\right)\right\Vert ^{2}-\left\Vert V^{-\frac{1}{2}}W+V^{\frac{1}{2}}t\left(\in A_{n}\right)\right\Vert ^{2}
\end{alignat*}
where $W$ is an $N\left(\boldsymbol{0},A\right)$ random vector.
\end{prop}

\begin{proof}
By Lemma \ref{lem:LLR} and simple algebra
\begin{alignat*}{1}
\Lambda_{n}= & \inf_{t\in A_{n}}H_{n}\left(\theta_{0},t\right)-\inf_{t\in A_{n,0}}H\left(\theta_{0},t\right)\\
= & 2\inf_{t\in A_{n}}\left(n\left\langle \frac{1}{\sqrt{n}}t,\partial F_{n}\left(\theta_{0}\right)\right\rangle +\frac{1}{2}\left\langle Vt,t\right\rangle \right)\\
 & -2\inf_{t\in A_{n,0}}\left(n\left\langle \frac{1}{\sqrt{n}}t,\partial F_{n}\left(\theta_{0}\right)\right\rangle +\frac{1}{2}\left\langle Vt,t\right\rangle \right)+o_{P}\left(1\right)\\
= & \left\Vert V^{-\frac{1}{2}}\mathbb{G}_{n}\partial\rho\left(\theta_{0}\right)+V^{\frac{1}{2}}\hat{t}\left(\in A_{n,0}\right)\right\Vert ^{2}-\left\Vert V^{-\frac{1}{2}}\mathbb{G}_{n}\partial\rho\left(\theta_{0}\right)+V^{\frac{1}{2}}\hat{t}\left(\in A_{n}\right)\right\Vert ^{2}+o_{P}\left(1\right)
\end{alignat*}
the proposition follows by the continuous mapping theorem.
\end{proof}

\begin{example*}[$L_{1}$ regression(continued)]
 Consider a likelihood ratio statistics for testing the value of
$\left\langle \theta_{0},x_{0}\right\rangle $ at any $x_{0}\in E$.
For some prespecified point $\left(x_{0},c\right)$, we consider the
following hypothesis:
\begin{eqnarray*}
H_{0}:\left\langle \theta_{0},x_{0}\right\rangle \leq0 & \textrm{vs.} & H_{1}:\left\langle \theta_{0},x_{0}\right\rangle >0.
\end{eqnarray*}
The objective function under the null constrained is defined as
\begin{alignat*}{1}
F_{n}\left(\theta^{H_{0}}\right)= & \frac{1}{n}\sum_{i=1}^{n}\left|y_{i}-\left\langle x_{i},\theta^{H_{0}}\right\rangle \right|+\frac{\lambda}{2}\left\Vert \theta^{H_{0}}\right\Vert 
\end{alignat*}
where $\theta^{H_{0}}\in H_{0}=\left\{ \theta\in\Theta:\left\langle \theta_{0},x_{0}\right\rangle \leq0\right\} $.
Note that the set $H_{0}$ is convex. We define the generalized likelihood
ratio test statistic as
\begin{alignat*}{1}
\Lambda_{n}= & F_{n}\left(\hat{\theta}^{H_{0}}\right)-F_{n}\left(\hat{\theta}_{n}\right),
\end{alignat*}
where $\hat{\theta}^{H_{0}}$ is the M-estimator under convex constraint:
\begin{alignat*}{1}
\hat{\theta}^{H_{0}}= & \arg\min_{\theta^{H_{0}}\in H_{0}}F_{n}\left(\theta^{H_{0}}\right).
\end{alignat*}
If the null the interior of the hypothesis $H_{0}$ contains the true
parameter $\theta_{0}$, the sequence of $\Lambda_{n}$ converges
to zero in distribution. This means that an error of the first kind
converges to zero under that the null hypothesis is true. If the true
parameter $\theta_{0}$ belongs to the boundary: $\left\langle \theta_{0},x_{0}\right\rangle =0$,
the sets $\sqrt{n}\left(\Theta_{0}-\theta_{0}\right)$ converge to
the $H_{0}=\left\{ \theta:\ \left\langle \theta,x_{0}\right\rangle \leq0\right\} $.
The sequence of $\Lambda_{n}$ converges in distribution to the distribution
of the square distance of a standard normal vector to the half-space
$V^{\frac{1}{2}}H_{0}=\left\{ \theta:\ \left\langle \theta,V^{-\frac{1}{2}}x_{0}\right\rangle \leq0\right\} $,
that is the distribution of $\left(W\lor0\right)^{2}$.
\end{example*}
\appendix

\section{Appendix}

\subsection{Proof of Subdifferential Calculus of $\rho=\left|y-\left\langle x,\theta\right\rangle \right|$}

Here we show the subdifferential calculus of $\rho=\left|y-\left\langle x,\theta\right\rangle \right|$.
We use the following lemma.
\begin{lem}
The subdifferential of $\left\Vert \theta\right\Vert =\left\langle \theta,\theta\right\rangle $
is $\partial\left\Vert \theta\right\Vert =\left\{ \theta\right\} ,\ \theta\in\mathscr{H}$
.\end{lem}
\begin{proof}
For $\theta\in\mathscr{H}$,
\begin{alignat*}{1}
\left\langle \eta,\theta\right\rangle -\left\langle \theta,\theta\right\rangle  & =\left\langle \eta-\theta,\theta\right\rangle ,\quad\eta\in\mathscr{H},
\end{alignat*}
then $\partial\left\Vert \theta\right\Vert =\left\{ \theta\right\} $.
\end{proof}

\begin{prop*}[Subdifferential Calculus of $\rho=\left|y-\left\langle x,\theta\right\rangle \right|$]
 The criterion function $\rho\left(\theta,Z\right)=\left|y-\left\langle x,\theta\right\rangle \right|$
is a proper l.s.c. convex function and has the subdifferential such
that
\begin{alignat*}{1}
\partial\rho\left(\theta,Z\right) & =\begin{cases}
\textrm{sgn}\left(y-\left\langle x,\theta\right\rangle \right)x, & \text{{if}}\ y-\left\langle x,\theta\right\rangle \neq0;\\
\left[-1,1\right]x, & \text{{if}}\ y-\left\langle x,\theta\right\rangle =0.
\end{cases}
\end{alignat*}
\end{prop*}
\begin{proof}
Let $t\in\left[-1,1\right]$, $\theta=tx$. For all $\zeta\in\mathscr{H}$,
\begin{alignat*}{2}
\left\langle tx,\zeta-\theta\right\rangle  & =t\left\langle x,\zeta\right\rangle -ty & \leq t\left|\left\langle x,\zeta\right\rangle -y\right| & \leq\left|t\right|\left|\left\langle x,\zeta\right\rangle -y\right|\leq\left|\left\langle x,\zeta\right\rangle -y\right|.
\end{alignat*}
Then, $\theta=tx\in\partial\rho\left(y-\left\langle x,\theta\right\rangle =0\right)$
and $\left[-1,1\right]x\subset\partial\rho\left(y-\left\langle x,\theta\right\rangle =0\right)$. 

Next, we shall show the inverse inclusion: $\partial\rho\left(y-\left\langle x,\theta\right\rangle =0\right)\subset\left[-1,1\right]x$.
Let $\theta\in\partial\rho\left(y-\left\langle x,\theta\right\rangle =0\right)$
and assume $\theta\neq x$. From $\theta\in\partial\rho\left(y-\left\langle x,\theta\right\rangle =0\right)$,
we have 
\begin{alignat}{1}
\left|y-\left\langle x,\zeta\right\rangle \right| & \geq\left\langle \zeta-\theta,\ \theta\right\rangle ,\qquad^{\forall}\zeta\in\mathscr{H}.\label{(*)}
\end{alignat}
From now on, set $H=\left\{ \eta\in\mathscr{H}:\ \left\langle x,\eta\right\rangle =y\right\} $
and $G=\left\{ \eta\in\mathscr{H}:\ \left\langle \eta,\theta\right\rangle =\left\langle \theta,\theta\right\rangle \right\} $,
we shall show that $H=G$. When $\dim\left(\mathscr{X}\right)=1$,
$H=G=\left\{ \frac{y}{x^{*}}\right\} $. Assume $\dim\left\{ \mathscr{H}\right\} >2$.
First $\eta\in H\Rightarrow\eta\in G$, pick $\eta\in H$: $\left\langle x,\eta\right\rangle =y$
we have $\eta=\theta$, so $\left\langle \eta,\theta\right\rangle =\left\langle \theta,\theta\right\rangle $.
Then, $H\subset G$. We shall show the inverse inclusion $G\subset H$.
Assume $\eta\in G$ and $\eta\notin H$. Because $\theta\neq x$,
there exists $u\in\mathscr{H}$ such that $\left\langle \theta,u\right\rangle \neq y$.
Put $p=\left\langle x,\eta\right\rangle u-\left\langle x,u\right\rangle \eta+\theta$,
because $u$ and $\eta$ are linear independent, $p\neq\theta$. On
the other hand
\begin{alignat*}{1}
\left\langle x,p\right\rangle  & =\left\langle x,\left\langle x,\eta\right\rangle u-\left\langle x,u\right\rangle \eta+\theta\right\rangle \\
 & =\left\langle x,\eta\right\rangle \left\langle x,u\right\rangle -\left\langle x,u\right\rangle \left\langle x,\eta\right\rangle +y\\
 & =y.
\end{alignat*}
This is contradiction, therefore $G\subset H$. Finally, we have $G=H$.

Now, set
\begin{alignat*}{1}
x^{\prime} & \triangleq\zeta-\frac{y-\left\langle x,\zeta\right\rangle }{y-\left\langle x,v\right\rangle }\left(v-\theta\right),\qquad^{\forall}\zeta\in\mathscr{H},
\end{alignat*}
Then, we have
\begin{alignat*}{1}
\left\langle x,x^{\prime}\right\rangle  & =\left\langle x,\zeta\right\rangle -\frac{y-\left\langle x,\zeta\right\rangle }{y-\left\langle x,v\right\rangle }\left\langle x,v-\theta\right\rangle \\
 & =\left\langle x,\zeta\right\rangle -\frac{y-\left\langle x,\zeta\right\rangle }{y-\left\langle x,v\right\rangle }\left(\left\langle x,v\right\rangle -y\right)\\
 & =\left\langle x,\zeta\right\rangle +y-\left\langle x,\zeta\right\rangle \\
 & =y.
\end{alignat*}
 Furthermore $x^{\prime}\in H\Rightarrow x^{\prime}\in G$. Therefore,
\begin{alignat*}{1}
\left\langle \theta,\theta\right\rangle  & =\left\langle \theta,x^{\prime}\right\rangle \\
 & =\left\langle \theta,\zeta\right\rangle -\frac{y-\left\langle x,\zeta\right\rangle }{y-\left\langle x,v\right\rangle }\left\langle \theta,\ v-\theta\right\rangle \\
 & =\left\langle \theta,\zeta\right\rangle -\frac{\left\langle \theta,\ v-\theta\right\rangle }{y-\left\langle x,v\right\rangle }\left(y-\left\langle x,\zeta\right\rangle \right)\\
 & =\left\langle \theta,\zeta-\theta\right\rangle -\frac{\left\langle \theta,\ v-\theta\right\rangle }{y-\left\langle x,v\right\rangle }\left(y-\left\langle x,\zeta\right\rangle \right)\\
 & =\left\langle \theta,\zeta-\theta\right\rangle -\frac{\left\langle \theta,\ v-\theta\right\rangle }{y-\left\langle x,v\right\rangle }\left(\left\langle x,\theta\right\rangle -\left\langle x,\zeta\right\rangle \right),
\end{alignat*}
and we get $\left\langle \theta,\zeta-\theta\right\rangle =t\left\langle x,\zeta-\theta\right\rangle $
where $t=\frac{\left\langle \theta,\ v-\theta\right\rangle }{y-\left\langle x,v\right\rangle }\neq0$.
Because of $\eqref{(*)}$, $\left\langle \theta,\ v-\theta\right\rangle \leq\left|y-\left\langle x,v\right\rangle \right|$
and 
\begin{alignat*}{1}
-\left\langle \theta,\ v-\theta\right\rangle =\left\langle \theta,\ \theta-v\right\rangle  & \leq\left|-\left\langle x,\theta-v\right\rangle \right|\\
 & =\left|\left\langle x,v\right\rangle -y\right|\\
 & =\left|y-\left\langle x,v\right\rangle \right|,
\end{alignat*}
Since $\left\langle \theta,\zeta-\theta\right\rangle \neq0$,$\left\langle x,\zeta-\theta\right\rangle \neq0$.
We have $\left|\left\langle \theta,\ v-\theta\right\rangle \right|\leq\left|y-\left\langle x,v\right\rangle \right|$,
$\left|t\right|\leq1$. Therefor, $\partial\rho\left(y-\left\langle x,\theta\right\rangle =0\right)\subset\left[-1,1\right]x$.
\end{proof}

\subsection{Proof of Proposition \ref{prop:3}}

Set the following notation;\\
$\left(\Omega,\mathcal{F},\mathbb{P}\right)$: probability triple\\
$\left(\mathscr{H},\mathcal{H}\right)$: real separable Hilbert space
with Borel $\sigma$-field\\
$2^{\mathscr{H}}$: the family of all nonempty subsets of $\mathscr{H}$\\
$F:\Omega\rightarrow2^{\mathscr{H}}$: set-valued function.\\
The inverse image $F^{-1}\left(X\right)$ is defined by
\begin{alignat*}{1}
F^{-1}\left(X\right) & =\left\{ \omega\in\Omega:\ F\left(\omega\right)\cap X\neq\textrm{Ø}\right\} .
\end{alignat*}
A set-valued function $F:\Omega\rightarrow2^{\mathscr{X}}$ is called
measurable if $F^{-1}\left(X\right)$ is measurable for every closed
subset $X$ of $\mathscr{X}$. For $1\leq p\leq\infty$ define a selection
of $F$ by
\begin{alignat*}{1}
S_{F}^{p} & =\left\{ f\in L_{p}\left[\Omega,\mathcal{F},\mu\right]:\ f\left(\omega\right)\in F\left(\omega\right)\ a.e.\left(\mu\right)\right\} .
\end{alignat*}
The key notion of set-valued mesurable mapping is decomposability.
\begin{defn}
\textbf{Decomposability} {[}Section3 in \cite{Hiai-Umegaki_77}{]}\\
Let $M$ be a set of measurable functions $f:\Omega\mapsto\mathscr{H}$.
$M$ is called decomposable with respect to $\mathcal{F}$ if $f_{1},f_{2}\in M$
and $A\in\mathcal{F}$ implies
\begin{alignat*}{1}
\mathbb{I}_{A}f_{1}+\mathbb{I}_{\Omega\setminus A}f_{2} & \in M.
\end{alignat*}

\end{defn}
For proof of Proposition \ref{prop:3}, we need lemmas from \cite{Hiai-Umegaki_77}.
\begin{lem}
\label{Hiai-Umegaki1.1}{[}Lemma1.1. in \cite{Hiai-Umegaki_77}\}\\
Let $F$ be measurable set-valued function. If $S_{F}^{p}$ is nonempty,
then there exists a sequence $\left\{ f_{n}\right\} \in S_{F}^{p}$
such that $F\left(\omega\right)=\textrm{cl}\left\{ f_{n}\left(\omega\right)\right\} $
for all $\omega\in\Omega$.
\end{lem}

\begin{lem}
\label{Hiai-Umegaki2.1.} {[}Lemma 2.1. in \cite{Hiai-Umegaki_77}{]}\\
Let $\phi:\Omega\times\mathscr{H}$ be $\mathcal{F}\otimes\mathcal{H}$-measurable.
Assume $\left(\Omega,\mathcal{F},\mathbb{P}\right)$ is complete and
$\phi\left(\omega,\theta\right)$ is l.s.c. in $\theta$ for every
fixed $\omega$. Then the function 
\begin{alignat*}{1}
\omega & \mapsto\inf\left\{ \phi\left(\omega,\theta\right):\theta\in F\left(\omega\right)\right\} ,
\end{alignat*}
 is measurable.
\end{lem}

\begin{lem}
\textup{\label{Hiai-Umegaki3.1.}} {[}Theorem 3.1. in \cite{Hiai-Umegaki_77}\}\\
$M=S_{F}$ if and only if $M$ is decomposable.
\end{lem}
For the set-valued random variables the following Theorem and definition
were given by Hiai and Umegaki

\begin{prop*}
\ref{prop:3} There is a measurable selector of subdifferential $\partial f$
i.e., $S_{\partial f}\neq\emptyset$. And the set of all measurable
selector is identical to subdifferential $\partial f$: $S_{\partial f}=\partial f$. \end{prop*}
\begin{proof}
Let $h\left(\gamma,z\right)$ as 
\begin{alignat*}{1}
h\left(\gamma,z\right) & =\inf_{\left|\beta-\alpha\right|\leq1}\left\{ f\left(\beta,z\right)-f\left(\alpha,z\right)-\left\langle \beta-\alpha,\gamma\right\rangle \right\} .
\end{alignat*}
Fix $\alpha$. $\gamma$ is a subdifferential of $f\left(\cdot,z\right)$
at $\alpha$ iff $h\left(\gamma,z\right)\geq0$. For every $z$, $h\left(\gamma,\cdot\right)$
is measurable. From Lemma \ref{Hiai-Umegaki2.1.} $\gamma\left(\cdot\right)$
is measurable.

Let $\gamma_{1}\left(\cdot\right)$ and $\gamma_{2}\left(\cdot\right)$
be measurable selector of subdifferential $\partial f\left(\alpha,\cdot\right)$
satisfying
\begin{alignat*}{1}
f\left(\beta,\cdot\right) & \geq f\left(\alpha,\cdot\right)-\left\langle \beta-\alpha,\gamma_{1}\left(\cdot\right)\right\rangle ,\\
f\left(\beta,\cdot\right) & \geq f\left(\alpha,\cdot\right)-\left\langle \beta-\alpha,\gamma_{2}\left(\cdot\right)\right\rangle .
\end{alignat*}
From the following inequality
\begin{alignat*}{1}
f\left(\beta,\cdot\right) & \geq f\left(\alpha,\cdot\right)-\left\langle \beta-\alpha,\mathbb{I}_{A}\left(\cdot\right)\gamma_{1}\left(\cdot\right)+\mathbb{I}_{\Omega\setminus A}\left(\cdot\right)\gamma_{2}\left(\cdot\right)\right\rangle ,
\end{alignat*}
$\partial f\left(\alpha,\cdot\right)$ is decomposable. Therefore,
from Lemma \ref{Hiai-Umegaki3.1.} and Lemma \ref{Hiai-Umegaki1.1},
$S_{\partial f}=\partial f$.
\end{proof}

\subsection{A Existence of Minimum}
\begin{prop*}
Existence of Minimum\\
Suppose $f:\Theta\rightarrow(-\infty,\infty]$ is a lower semi-continuous
convex (l.s.c.) functional and its domain $\Theta$ is bounded. Then
there exists $\arg\min_{\theta}f\left(\omega,\theta\right)$ and $\inf_{\theta}f\left(\omega,\theta\right)$.\end{prop*}
\begin{proof}
Let $C$ be a convex subset of a Banach space. From the separation
theorem, $C$ is closed in norm topology if and only if $C$ is closed
in the weak topology(Correspondence of closedness). $f$ is lsc on
$\Theta$ in the norm topology if and only if $f$ is lsc in the weak
topology.

For each $a\in\mathbb{R}$ put 
\begin{align*}
G_{a} & =\left\{ \theta\in\Theta:\ f\left(\theta\right)>a\right\} .
\end{align*}
$G_{a}$ is open in the weak topology and $\Theta=\bigcup_{a\in\mathbb{R}}G_{a}$.
Since $\Theta$ is weakly compact, there is finite subcover such that
\begin{align*}
\Theta & =\bigcup_{i=1}^{n}G_{a_{i}}.
\end{align*}
Putting $a_{0}=\min\left\{ a_{1},\cdots,a_{n}\right\} $, we have
$f\left(\theta\right)>a_{0}$ for all $\theta\in\Theta$. There exists
a real number $b=\inf\left\{ f\left(\theta\right):\ \theta\in\Theta\right\} $.
\\
Suppose $f\left(\theta\right)>b$ for all $\theta\in\Theta$, then
\begin{align*}
\Theta & =\bigcup_{n=1}^{\infty}\left\{ \theta:\ f\left(\theta\right)>b+\frac{1}{n}\right\} .
\end{align*}
Since $\Theta$ is weakly compact,
\begin{align*}
\Theta & =\bigcup_{i=1}^{m}\left\{ \theta:\ f\left(\theta\right)>b+\frac{1}{n_{i}}\right\} .
\end{align*}
Put $b_{0}=\min\left\{ b+\frac{1}{n_{1}},\cdots,b+\frac{1}{n_{m}}\right\} $,
we have $f\left(\theta\right)>b_{0}$ for all $\theta$. Therefore
we have 
\begin{align*}
b & =\inf\left\{ f\left(\theta\right):\ \theta\in\Theta\right\} \geq b_{0}>b.
\end{align*}
This is a contradiction.
\end{proof}

\section*{Acknowledgements}

This research is supposed by grant-in-aid for JSPS Fellows (DC1, 20137989).

\bibliographystyle{ecta}
\bibliography{Mosco_2}

\end{document}